\documentclass[a4paper,letterpaper,leqno]{amsart}
\usepackage{amsmath, amsthm, amsfonts}
\usepackage{amssymb}
\usepackage{latexsym}
\usepackage{verbatim}
\usepackage{url}
\usepackage{fullpage}
\usepackage[all,cmtip]{xy}
\usepackage{color}
\usepackage[usenames,dvipsnames]{xcolor}
\usepackage{hyperref}
\hypersetup{colorlinks=true,urlcolor=MidnightBlue,citecolor=PineGreen,linkcolor=BrickRed}

\usepackage{enumerate}
\usepackage{amssymb}
\usepackage{array}

\DeclareMathAlphabet{\mathfr}{U}{euf}{m}{n}

\newtheorem{theorem}{Theorem}[section]
\newtheorem{conjecture}[theorem]{Conjecture}
\newtheorem{proposition}[theorem]{Proposition}

\newtheorem{lemma}[theorem]{Lemma}

\theoremstyle{definition}

\newtheorem{remark}[theorem]{Remark}

\numberwithin{equation}{section}

\newcommand{\C}{\mathbb C}
\newcommand{\F}{\mathbb F}

\newcommand{\Q}{\mathbb Q}

\newcommand{\R}{\mathbb R}
\newcommand{\Z}{\mathbb Z}

\newcommand{\p}{\mathfrak p}

\newcommand{\triv}{\mathrm{triv}}
\newcommand{\ntriv}{\mathrm{ntriv}}

\newcommand{\USp}{\mathrm{USp}}
\newcommand{\Frob}{\mathrm{Frob}}

\newcommand{\Os}{\mathcal O}

\newcommand{\Sp}{\operatorname{Sp}}

\newcommand{\Tr}{\operatorname{Trace}}

\newcommand{\Aut}{\operatorname{Aut}}
\newcommand{\Li}{\operatorname{Li}}
\newcommand{\ST}{\operatorname{ST}}

\newcommand{\NN}{\operatorname{Nm}}

\newcommand{\Hom}{\operatorname{Hom}}
\newcommand{\GL}{\operatorname{GL}}

\newcommand{\Sym}{\operatorname{Sym}}

\newcommand{\Zar}{\operatorname{Zar}}
\newcommand{\AST}{\operatorname{AST}}

\begin{document}
\title{Frobenius sign separation for abelian varieties}

\author{Alina Bucur}
\address{Department of Mathematics \\ University of California, San Diego \\ 9500 Gilman Drive \#0112 \\ 
La Jolla \\ CA 92093 \\ USA}
\email{alina@math.ucsd.edu}
\urladdr{https://www.math.ucsd.edu/~alina/}

\author{Francesc Fit\'e}
\address{Departament de matem\`atiques i inform\`atica,
Universitat de Barce\-lona,
Gran via de les Corts Catalanes 585, 08007 Barcelona, Catalonia, Spain}
\email{ffite@ub.edu}
\urladdr{http://www.ub.edu/nt/ffite/}

\author{Kiran S. Kedlaya}
\address{Department of Mathematics \\ University of California, San Diego \\ 9500 Gilman Drive \#0112 \\ 
La Jolla \\ CA 92093 \\ USA}
\email{kedlaya@ucsd.edu}
\urladdr{http://kskedlaya.org}



\begin{abstract}
Let $A$ and $A'$ be nonzero abelian varieties defined over a number field $k$ such that $\Hom(A,A')=0$. Under the Generalized Riemann hypothesis for motivic $L$-functions attached to $A$ and $A'$, we show that there exists a prime $\p$ of $k$ of good reduction for $A$ and $A'$ at which the Frobenius traces of $A$ and $A'$ are nonzero and differ by sign, and such that the norm $\NN(\p)$ is $O_{k,g,g'}(\log(2NN')^2)$, where $N$ and $N'$ respectively denote the absolute conductors of $A$ and $A'$. We also make the dependence of the big-$O$ constant on $k$ and the dimensions $g,g'$ of $A,A'$ explicit up to an effectively computable absolute constant.
Our method extends that of Chen, Park, and Swaminathan who considered the case in which $A$ and $A'$ are elliptic curves.
\end{abstract}

\maketitle

\section{Introduction}\label{section: introduction}

Let $A$ and $A'$ be nonzero abelian varieties defined over a number field $k$ of respective positive dimensions $g$ and $g'$, with absolute conductors $N$ and $N'$. If $A$ and $A'$ are elliptic curves not isogenous over $k$, an argument of Serre (see \cite[Cor. 23]{BK16b}) provides a prime $\p$ of $k$ with norm $\NN(\p)=O_k(\log(2NN')^2)$ of good reduction for both $A$ and $A'$ at which the respective Frobenius traces $a_\p(A)$ and $a_\p(A')$ are distinct. In this note we investigate the more refined question of the existence of one such prime $\p$ for which $a_\p(A)\cdot a_\p(A')<0$. 

Under certain restrictions on the geometry of $A$ and $A'$, the implications for this problem of the effective Sato--Tate conjecture were considered in \cite[Cor. 1.3]{BFK23} (extending \cite[Thm. 4.3]{BK16b} in the case of nonisogenous elliptic curves without complex multiplication); see below for a closer discussion of these results. 

After \cite{BK16b}, the Frobenius sign separation problem for nonisogenous elliptic curves was reexamined by Chen, Park, and Swaminathan \cite[Thm. 1.3]{CPS18} who improved on the upper bound for a sign separating prime. Their method does not formally assume the effective Sato--Tate conjecture, but it is close in spirit: it is conditional on the generalized Riemann hypothesis for certain symmetric power $L$-functions. 

Let $\p$ be a prime of good reduction for $A$ and $A'$. Conveniently adapted to the higher dimensional setting, one key idea in \cite{CPS18} consists in rewriting the condition $a_\p(A)\cdot a_\p(A')<0$ as
$$
\psi_\p:= a_\p(A) \cdot a_\p(A')\cdot  (-2g + a_\p(A)) \cdot (2g' + a_\p(A'))>0.
$$
In the case of nonisogenous elliptic curves, Chen, Park, and Swaminathan determine the asymptotic \emph{strictly positive} value of the (conveniently weighted) sum of the $\psi_\p$, as well as the error term.
For the latter, they crucially use a kernel introduced by Bach \cite{Bac90} which allows them to optimize the error term not asymptotically, but where it is first dominated by the main term; this ultimately yields a savings of a log-log factor over \cite{BK16b}.

Let $\ST(A)$ and $\ST(A')$ be the Sato-Tate groups of $A$ and $A'$, together with their tautological representations $V$ and $V'$. Let $\ST(A\times A')$ denote the Sato--Tate group of the product $A\times A'$. It is the consideration of these Sato--Tate groups that forces us to assume that the Mumford--Tate conjecture holds for $A$ and $A'$ (see \S\ref{section: conjectures}), which we will do from now on. Our starting point is the reinterpretation of the quantities $\psi_\p$ in terms of Frobenius traces of the virtual representation 
\begin{equation}\label{equation: virtualrep}
\ST(A\times A') \rightarrow \GL((V^{\oplus -2g} \oplus V\otimes V)\otimes (V'^{\oplus 2g'} \oplus V'\otimes V')).
\end{equation} 
We show that if $\Hom(A,A')=0$, then the multiplicity of the trivial representation in the above virtual representation is strictly positive (see Lemma \ref{lemma: trivmult}), which provides a conceptual explanation of the positivity of the asymptotic value in \cite{CPS18}.

We will denote by $L(s,\chi)$ the (normalized) $L$-function attached to an irreducible character $\chi$ of $\ST(A\times A')$. 
The following theorem is the main contribution of this note.

\begin{theorem}\label{theorem: Main}
Let $A$ (resp. $A'$) be an abelian variety defined over $k$ of dimension $g$ (resp. $g'$), and with absolute conductor $N$ (resp. $N'$).
Suppose that the Mumford--Tate conjecture holds for $A$ and $A'$, and that the generalized  Riemann hypothesis holds for $L(s,\chi)$, where $\chi$ is an irreducible constituent of the representation \eqref{equation: virtualrep}. If $\Hom(A,A')=0$, then there exists a prime $\p$ not dividing $NN'$ with norm 
$$
\NN(\p)=O([k:\Q]^2\log(2|\Delta_k|)^2g^4(g')^4(g+g')^2\log(2NN')^2)
$$ 
such that $a_\p(A)$ and $a_\p(A')$ are nonzero and of opposite sign. The implied constant in the $O$-notation is an absolute constant (independent of $A$, $A'$, $k$, $g$, and $g'$). 
\end{theorem}

The above result improves on the bound 
$$
\NN(\p)=O_{k,g,g'}\left(\log(2NN')^2\log(\log(4NN'))^6 \right)
$$
obtained in \cite[Cor. 1.3]{BFK23} in three different ways. First, it removes the $\log(\log(4NN'))^6$ factor. Second, it makes explicit the dependence of the implied constant on $k$, $g$, and $g'$. Third, it relaxes the hypotheses of \cite[Cor. 1.3]{BFK23} in several directions. To begin with, it replaces the condition $\ST(A\times A')\simeq \ST(A)\times \ST(A')$ with the weaker hypothesis that $\Hom(A,A')=0$, and it no longer requires the connectedness of $\ST(A)$ and $\ST(A')$. Furthermore, the assumption of the generalized Riemann hypothesis is limited to a finite number of irreducible representations of $\ST(A\times A')$, in contrast with its assumption for all irreducible representations of $\ST(A)\times \ST(A')$ made in \cite{BFK23}. 

In Section  \S\ref{section: conjectures}, we set notation around Sato--Tate groups and $L$-functions, and formulate the form of the generalized Riemann hypothesis on which our results depend (Conjecture~\ref{conjecture: GRH}).

In Section \S\ref{section: Bach}, we take the adaptation of Bach's argument
made in \cite{CPS18} for elliptic curves, and further adapt it to give a corresponding estimate for an arbitrary abelian variety (Proposition~\ref{proposition: truncsum}).

In Section \S\ref{section: linsignrevis}, we apply the adaptation of Bach's argument to prove Theorem~\ref{theorem: Main}.

\subsection*{Notation and terminology} 
Throughout this article, $k$ is a fixed number field and $g$ and $g'$ are fixed positive integers. 
By a prime of $k$, we refer to a nonzero prime ideal of the ring of integers of $k$.

For an ordered set $(X,\leq)$ and functions $f,h\colon X\rightarrow \R$ we write $f(x)=O(h(x))$ to indicate that there exist a real number $M>0$ and an element $x_0 \in X$ such that $|f(x)|\leq Mh(x)$ for every $x\geq x_0$. We will generally specify the element $x_0$ in the statements of theorems, but we will usually obviate it in their proofs, where it can be inferred from the context. We refer to $M$ as the \emph{implied constant} in the $O$-notation. In the results of this note, the implied constant is an \emph{effectively} computable absolute constant; we will not make this explicit here, but see \cite[(5.15)]{CPS18} for the case $g=g'=1$.

\subsection*{Acknowledgements} All three authors were supported by the Institute for Advanced Study during 2018--2019; this includes funding from National Science Foundation grant DMS-1638352. All three authors were additionally supported by the Simons Foundation grant 550033, and by the Hausdorff Research Institute for Mathematics funded by the Deutsche Forschungsgemeinschaft under Germany's Excellence Strategy--EXC-2047/1--390685813. Bucur was also supported by the Simons Foundation collaboration grant 524015, and by NSF grants DMS-2002716 and DMS-2012061. Fité was additionally supported by the Ram\'on y Cajal fellowship RYC-2019-027378-I, by the Mar\'ia de Maeztu program CEX2020-001084-M, by the AEI grant PID2022-137605NB-I00, and by the ERC grant 682152. Kedlaya was additionally supported by NSF grants DMS-1501214, DMS-1802161, DMS-2053473, DMS-2401536, by the UCSD Warschawski Professorship, and by a Simons Foundation fellowship during 2023--2024.

\section{Conjectural framework}\label{section: conjectures}

In this section, we present the conjectural framework on which the results of this article rely. Let $A$ denote an abelian variety of dimension~$g$ over $k$ of absolute conductor $N:=N_A$, that is, the absolute norm of the conductor ideal of $A$. For a rational prime $\ell$, we denote by 
$$
\varrho_{A,\ell}\colon G_k \rightarrow \Aut(V_\ell(A))
$$
the $\ell$-adic representation attached to $A$, obtained from the action of the absolute Galois group of $k$ on the rational $\ell$-adic Tate module $V_\ell(A):=T_\ell(A)\otimes \Q_\ell$. Let $\p$ be a prime  of $k$. If $\p$ does not divide~$N\ell$, write  $a_\p:=a_\p(A)$ for the trace of $\varrho_{A,\ell}(\Frob_\p)$, where $\Frob_\p$ is a Frobenius element at $\p$.

As explained in \cite[\S2.1]{BFK23}, under the assumption of the Mumford--Tate conjecture for $A$, one can construct an algebraic group $\AST(A)$ over $\Q$ with the property that for every prime $\ell$,
$$
\AST(A)\times _\Q \Q_\ell \simeq G_\ell^{1,\Zar}.
$$
Here $G_\ell^{\Zar}$ denotes the Zariski closure of the image of $\varrho_{A,\ell}$, and $G_\ell^{1,\Zar}$ denotes the intersection of  
$G_\ell^{\Zar}$ with $\Sp_{2g}/\Q_\ell$. The Sato--Tate group $\ST(A)$ is defined as a maximal compact subgroup of the group of $\C$-points of $\AST(A)$. We may identify $\ST(A)$ with its image under its tautological representation 
$$
\varrho\colon \ST(A)\rightarrow \GL(V),
$$
where $V$ is a $2g$-dimensional $\C$-vector space, which is unitary and symplectic. Hence, $\ST(A)$ can be regarded as a compact real Lie subgroup of $\USp(2g)$. As described in \cite[\S8.3.3]{Ser12}, if $\p$ does not divide $N\ell$, one can construct an element $y_\p$ in the set of conjugacy classes $Y$ of $\ST(A)$ with the property that
$$
\det(1-\varrho_{A,\ell}(\Frob_\p)\NN(\p)^{-1/2}T)=\det(1-\varrho(y_\p) T).
$$ 
In particular, if we denote by 
$$
\overline a_\p:=\frac{a_\p}{\sqrt{\NN(\p)}}
$$ 
the normalized Frobenius trace, we have that $\Tr(y_\p)=\overline a_\p$.
More in general, via Weyl's unitarian trick, any complex representation 
$$
\sigma\colon \ST(A) \rightarrow \GL(V_\chi)
$$ 
of character $\chi$ and degree $d_\chi$ gives rise to a representation
$$
\sigma_{A,\ell}\colon G_k\rightarrow \Aut(V_{\chi,\ell})
$$
where $V_{\chi,\ell}$ is a $\overline \Q_\ell$-vector space of dimension $d_\chi$. If $\p$ does not divide $N\ell$, one has
$$
\det(1-\sigma_{A,\ell}(\Frob_\p)\NN(\p)^{-w_\chi/2}T)=\det(1-\sigma(y_\p) T),
$$
where $w_\chi$ denotes the motivic weight of $\chi$. For a general $\p$, define
$$
L_\p(T, \chi):=\det(1-\sigma_{A,\ell}(\Frob_\p)\NN(\p)^{-w_\chi/2}T\,|\,V_{\chi,\ell}^{I_\p}),
$$
where $I_\p$ denotes the inertia subgroup of the decomposition group $G_\p$ at $\p$. The polynomials $L_\p(T, \chi)$ do not depend on $\ell$, and have degree $d_\chi(\p) \leq d_\chi$. Moreover, writing $\alpha_{\p,j}$ for $j=1,\dots,d_\chi(\p)$ to denote the reciprocal roots of $L_\p(T,\chi)$, we have that 
$$
|\alpha_{\p,j}|\leq 1.
$$
In fact, if $\p$ does not divide $N$, we have that $d_\chi(\p) = d_\chi$ and $|\alpha_{\p,j}|= 1$.
Therefore, the Euler product 
$$
L(s,\chi):=\prod_{\p}L_\p(\NN(\p)^{-s},\chi)^{-1}
$$
is absolutely convergent for $\Re(s)>1$. We will make strong assumptions on the analytic behavior of the above Euler product.

Following \cite[\S4.1]{Ser69}, define the positive integer 
$$
B_\chi:=|\Delta_k|^{d_\chi}\cdot N_\chi,
$$
where $\Delta_k$ is the absolute discriminant of $k$ and $N_\chi$ is the absolute conductor attached to the $\ell$-adic representation $\sigma_{A,\ell}$. For $j=1,\dots,d_\chi$, let $0\leq \kappa_{\chi,j}\leq 1+w_\chi/2$ be the local parameters at infinity (they are semi-integers that can be explicitly computed from the discussion in \cite[\S3]{Ser69}). Define the completed $L$-function
\begin{equation}\label{equation: LandGamma}
\Lambda(s,\chi):=B_\chi^{s/2}L(s,\chi)\Gamma(s,\chi),\qquad  \text{where}\quad\Gamma(s,\chi):=\pi^{d_\chi s/2}\prod_{j=1}^{d_\chi}\Gamma\left(\frac{s+\kappa_{\chi,j}}{2}\right).
\end{equation}
Let $\delta(\chi)$ be the multiplicity of the trivial representation in the character $\chi$ of $\ST(A)$.
\begin{conjecture}[Generalized Riemann hypothesis]\label{conjecture: GRH}
For every irreducible character $\chi$ of $\ST(A)$, the following holds:
\begin{enumerate}[i)]
\item The function $s^{\delta(\chi)}(s-1)^{\delta(\chi)}\Lambda(s,\chi)$ extends to an analytic function on $\C$ of order $1$ which does not vanish at $s=0,1$.
\item There exists $\epsilon \in \C^\times$ with $|\epsilon|=1$ such that for all $s\in \C$ we have
$$
\Lambda(s,\chi)=\epsilon\Lambda(1-s,\overline \chi),
$$
where $\overline \chi$ is the character of the contragredient representation of $\sigma$.
\item The zeros $\rho$ of $\Lambda(s,\chi)$ (equivalently, the zeros $\rho$ of $L(s,\chi)$ with $0<\Re(\rho)<1$) all have $\Re(\rho)=1/2$.
\end{enumerate}
\end{conjecture}

\section{Use of the Bach kernel}\label{section: Bach}

Let $\chi$ be an irreducible character of $\ST(A)$. In \cite{BFK23} we used the estimate
\begin{equation}\label{equation: Murtyestimate}
\sum_{\NN(\p)\leq x} \chi(y_\p)= \delta(\chi) \Li(x) + O(d_\chi \sqrt{x} \log(N(x+w_\chi)))\qquad\text{for }x\geq 2,
\end{equation}
where $\Li(x):=\int_2^\infty dt/\log(t)$, to derive an effective form of the Sato-Tate conjecture\footnote{
In \eqref{equation: Murtyestimate} and thereafter, we make the convention that  all sums involving the classes $y_\p$ run over primes $\p$ not dividing $N$ unless otherwise specified. A similar convention applies for sums involving the normalized Frobenius traces $\overline a_\p=\Tr(y_\p)$.
}. This estimate is obtained by computing a contour integral of the function 
$$
-\frac{L'}{ L}(s,\chi)\frac{x^s}{s}=\sum_{\p}\sum_{r\geq 1}\sum_{j=1}^{d_{\chi}(\p)}\alpha_{\p,j}^r\log(\NN(\p))\frac{(x/\NN(\p)^r)^s}{s}
$$
and by exploiting the fact that for a real number $c>1$ we have
\begin{equation}\label{equation: LOkernel}
\frac{1}{2\pi i}\int_{c-i\infty}^{c+i\infty}\frac{y^s}{s}ds=\begin{cases} 
0 & \text{if } 0 < y < 1,\\
1/2 & \text{if } y = 1,\\
1 & \text{if } y>1.
\end{cases}
\end{equation} 

Following Chen, Park, and Swaminathan \cite{CPS18} in the elliptic curve case, in this section we will obtain an analogue of \eqref{equation: Murtyestimate}
by replacing the use of the kernel in \eqref{equation: LOkernel} by that of
\begin{equation}\label{equation: Bkernel}
\frac{1}{2\pi i}\int_{2-i\infty}^{2+i\infty}\frac{y^s}{(s+a)^2}ds=\begin{cases}
0 & \text{if } 0<y<1,\\
y^{-a}\log (y) & \text{if } y\geq 1.
\end{cases}
\end{equation}
Here $a$ is any real number lying in $(0,1)$. This kernel was introduced by Bach (see \cite[Lem. 4.1]{Bac90}). 

\begin{remark}\label{remark: propvirtual}
It will be useful to derive a result that applies to a virtual character $\chi$ of $\ST(A)$, that is, an integral linear combination of irreducible characters of $\ST(A)$. We say that that $\chi$ is selfdual if $\chi=\overline\chi$. If $\chi=\sum_{i=1}^r n_i\chi_i$, where $n_i\in \Z$ and $\chi_i$ are irreducible characters of $\ST(A)$, the completed $L$-function attached to $\chi$ is defined as the product
$$
\Lambda(s,\chi):=\prod_{i=1}^r \Lambda(s,\chi_i)^{n_i}.
$$  
One can similarly define a conductor $B_\chi$ and a $\Gamma$-factor at infinity $\Gamma(s,\chi)$ for $\chi$.  By the weight $w_\chi$ of $\chi$ we mean the maximum of the motivic weights of the irreducible constituents of $\chi$. For a virtual character $\chi=\chi_+-\chi_{-}$, where $\chi_+$ and $\chi_-$ are honest characters without common irreducible constituents, define 
\begin{itemize}
\item $d_\chi=d_{\chi_+} + d_{\chi_-}$.
\item $\delta(\chi)=\delta(\chi_+)-\delta(\chi_-)$.
\item $t_\chi$ as the maximum of $\chi$ attained on $\ST(A)$ (which exists as $\ST(A)$ is compact).
\item $\gamma_\chi$ as the number of factors of the form $\Gamma((s+\kappa)/2)$, for $0\leq \kappa \leq w_\chi/2+1$,  that constitute $\Gamma(s,\chi)$ (multiplying or dividing)\footnote{In other words, $\gamma_\chi$ is the number of $\Gamma$-factors after eventual cancellation. See Lemma~\ref{lemma: auxdims} for an example of $\Gamma$-factor cancellation. This will later be exploited in Proposition~\ref{proposition: truncsum}.}.  
\item $n_\chi=\max\{t_\chi,\gamma_\chi\}$.
\end{itemize}
Note that $t_\chi=\gamma_\chi=n_\chi=d_\chi$ when $\chi$ is a honest character. In general, however, we only have $t_\chi,\gamma_\chi,n_\chi\leq d_\chi$. 
 
\end{remark}

Going back to the argument using Bach's kernel, we fix $a=1/4$ in \eqref{equation: Bkernel}. In fact, we could take any value $a\in (0,1/4]$, the particular choice of $a$ affecting only the implied constant in the $O$-notation in the subsequent results. In order to lighten the notation, for $r\geq 1$, let us set
$$
\Lambda_\chi(\p^r,x):=\left(\sum_{j=1}^{d_{\chi_{+}}(\p)}\alpha_{\p,j,+}^r -\sum_{j=1}^{d_{\chi_{-}}(\p)}\alpha_{\p,j,-}^r\right)\log(\NN(\p))\left(\frac{\NN(\p)^r}{x}\right)^a\log\left(\frac{x}{\NN(\p)^r}\right).
$$
Here, $\alpha_{\p,j,+}$ (resp. $\alpha_{\p,j,-}$) denote the reciprocal roots of $L_\p(T,\chi_+)$ (resp. $L_\p(T, \chi_-)$).
\begin{proposition}\label{proposition: truncsumweak}
Let $\chi$ be a virtual selfdual character of $\ST(A)$. Assuming Conjecture~\ref{conjecture: GRH} for each of the irreducible constituents of $\chi$, we have
$$
\sum_{\NN(\p)\leq x} \Lambda_\chi(\p,x)=\frac{16}{25}\delta(\chi)x+O((n_\chi+\log(B_\chi))\log(2+w_\chi)\sqrt x(\log x)^2)\qquad \text{for every $x\geq 2$}.
$$
\end{proposition}

\begin{proof}
By logarithmically differentiating, integrating, and applying \eqref{equation: Bkernel}, we obtain
\begin{equation}\label{equation: logdifint}
I_\chi=-\frac{1}{2\pi i}\int_{2-i\infty}^{2+i\infty} \frac{L'}{L}(s,\chi)\frac{x^s}{(s+a)^2}ds=\sum_{r\geq 1}\sum_{\NN(\p)^r\leq x}\Lambda_\chi(\p^r,x).
\end{equation}

Observe that
\begin{equation}\label{equation: roughbound}
\sum_{r\geq 2} \sum_{\NN(\p)^r\leq x} \Lambda_\chi(\p^r,x)\leq n_\chi \sum_{r\geq 2}\sum_{\NN(\p)\leq x^{1/r}}\log(\NN(\p)) \log x=O\left([k:\Q]n_\chi \sqrt x (\log x)^2\right).
\end{equation}
For the last equality we have used that, since there is no prime ideal $\p$ of $k$ such that $\NN(\p)^r\leq x$ when $r>\log_ 2(x)$, the sum indexed by $r$ in the above expression has at most $O(\log x)$ summands. We have also used the prime number theorem to bound the number of primes $\p$ with $\NN(\p)\leq x^{1/r}$ by $O([k:\Q]\sqrt x/\log(x))$. 
We thus have that
\begin{equation}\label{equation: intform}
\sum_{\NN(\p)\leq x} \Lambda_\chi(\p,x)=I_\chi+O([k:\Q]n_\chi\sqrt x(\log x)^2)
\end{equation}
By letting $U>0$ (resp. $T>0$) be a large real number such that $-U$ (resp. $\pm T$) does not coincide with any of the trivial zeros (resp. ordinates of the nontrivial zeros) of $L(s,\chi)$, one can show that
\begin{equation}\label{equation: contourint}
I_\chi=\lim_{T,U\rightarrow \infty}\frac{-1}{2\pi i}\int_{\Gamma_{T,U}}\frac{L'}{L}(s,\chi)\frac{x^s}{(s+a)^2}ds,
\end{equation}
where $\Gamma_{T,U}$ is the rectangular contour of vertices $2-iT$, $-U-iT$, $-U+iT$ and $2+iT$. By the residue theorem, we are left with estimating the absolute value of the residues of the integrand in \eqref{equation: contourint}. 

The pole at $s=1$ contributes a residue of
$$
\delta(\chi)\frac{x}{(1+a)^2}=\frac{16}{25}\delta(\chi)x.
$$
Let $Z_\chi^{\triv}$ (resp. $Z_\chi^{\ntriv}$) denote the collection of trivial (resp. nontrivial) zeros of $L(s,\chi)$ counted with multiplicities. The contribution of the residues at $Z_\chi^{\triv}$ satisfies 
$$
\left| \sum_{\rho\in Z_\chi^\triv}\frac{x^\rho}{(\rho+a)^2}\right|\leq \frac{n_\chi}{a^2}+n_\chi\sum_{k=1}^\infty \frac{x^{-k/2}}{(-k/2+a)^2}=O\left(n_\chi\int_{1/2}^\infty x^{-t}dt\right)=O\left(\frac{n_\chi}{\sqrt x} \right),
$$ 
where, in the first equality above, we have used the Cauchy--Schwarz inequality. We now proceed to estimate the contribution
$$
\left| \sum_{\rho\in Z_\chi^\ntriv}\frac{x^\rho}{(\rho+a)^2}\right|=\sqrt x \sum_{\rho\in Z_\chi^\ntriv}\frac{1}{|\rho+a|^2}=\frac{\sqrt x}{2a+1} \sum_{\rho\in Z_\chi^\ntriv}\left(\frac{1}{1+a-\rho}+\frac{1}{1+a-\overline \rho}\right)
$$
of the residues at $Z_\chi^\ntriv$. Our next goal is to bound the right-hand side of the above equation. By \cite[Thm. 5.6]{IK04}, together with the first paragraph of the proof of \cite[Prop. 5.7]{IK04} and the fact that $\chi$ is selfdual, we obtain
\begin{equation}\label{equation: logarithmicHadam}
-\frac{L'}{L}(s,\chi)=\frac{1}{2}\log(B_\chi)+ \frac{\delta(\chi)}{s} + \frac{\delta(\chi)}{s-1}+\frac{\Gamma'}{\Gamma}(s,\chi)-\frac{1}{2}\sum_ {\rho\in Z_\chi^\ntriv}\left(\frac{1}{s-\rho}+\frac{1}{s-\overline \rho} \right).
\end{equation}
Note that
\begin{equation}\label{equation: boundL}
\left|\frac{L'}{L}(1+a,\chi)\right|\leq n_\chi [k:\Q]\left| \frac{\zeta'}{\zeta}(1+a)\right|=O([k:\Q]n_\chi). 
\end{equation}
Using the identities
\begin{equation}\label{equation: gammaperiod}
\frac{\Gamma'}{\Gamma}(s)=O(\log (s))\text{ for $s\in \R_{>1}$}\qquad\text{and}\qquad \frac{\Gamma'}{\Gamma}(1+s)=\frac{1}{s}+\frac{\Gamma'}{\Gamma}(s)
\end{equation}
on \eqref{equation: LandGamma}, we find that the bound $0\leq \kappa_{\chi,j}\leq 1+w_\chi/2$ gives
\begin{equation}\label{equation: gammabound}
\left|\frac{\Gamma'}{\Gamma}(1+a,\chi)\right|\leq n_\chi \frac{\log(\pi)}{2}+ \sum_{j=1}^{\gamma_\chi}\left(\left|\frac{\Gamma'}{\Gamma}\left(\frac{1+a+\kappa_{\chi,j}}{2}+1\right)\right|+\frac{2}{1+a+\kappa_{\chi,j}}\right)=O(n_\chi\log(2+w_\chi)).
\end{equation} 
By evaluating \eqref{equation: logarithmicHadam} at $s=1+a$, and using \eqref{equation: boundL} and \eqref{equation: gammabound}, we obtain
\begin{equation}\label{equation: ntrivres}
\left| \sum_{\rho\in Z_\chi^\ntriv}\frac{x^\rho}{(\rho+a)^2}\right|=O\big([k:\Q](n_\chi+\log(B_\chi))\log(2+w_\chi)\sqrt x\big).
\end{equation}
It remains to estimate the contribution of the residue at $s=-a$, which is
$$
R_\chi=\left(\frac{L'}{L}(-a,\chi)\log x+\left(\frac{L'}{L}\right)'(-a,\chi) \right)x^{-a}.
$$

Since $s=-a$ lies outside the region of absolute convergence of the defining Euler product for $L(s,\chi)$, we cannot utilize the argument in \eqref{equation: boundL} to bound $(L'/L)(-a,\chi)$. To derive such a bound, it will be enough to consider the difference between \eqref{equation: logarithmicHadam} evaluated at $s=2$ and $s=-a$. Indeed, proceeding as in \eqref{equation: boundL} and \eqref{equation: gammabound}, one obtains that
$$
\left|\frac{L'}{L}(2,\chi)\right|=O([k:\Q]n_\chi),\qquad \left|\frac{\Gamma'}{\Gamma}(-a,\chi)\right|=O(n_\chi\log(2+w_\chi)), \qquad \left|\frac{\Gamma'}{\Gamma}(2,\chi)\right|=O(n_\chi\log(2+w_\chi)).
$$
One also easily obtains that 
$$
\sum_{\rho\in Z_\chi^\ntriv}\left|\frac{1}{2-\rho}-\frac{1}{-a-\rho}\right| \leq (a+2)\sum_{\rho\in Z_\chi^\ntriv}\frac{1}{|\rho+a|^2}=O([k:\Q](n_\chi+\log(B_\chi))\log(2+w_\chi)),
$$
where the last equality has been seen while proving \eqref{equation: ntrivres}. By looking at the difference of \eqref{equation: logarithmicHadam} evaluated at $s=-a$ and $s=2$, we finally get
\begin{equation}\label{equation: boundLma} 
\left|\frac{L'}{L}(-a,\chi)\right|=O([k:\Q](n_\chi+\log(B_\chi))\log(2+w_\chi)).
\end{equation}
In order to bound $(L'/L)'(-a,\chi)$, we differentiate \eqref{equation: logarithmicHadam} and evaluate at $s=-a$. We obtain
\begin{equation}\label{equation: boundLpa}
-\left(\frac{L'}{L}\right)'(-a,\chi)=-\frac{\delta(\chi)}{a^2}-\frac{\delta(\chi)}{(1+a)^2}+\left(\frac{\Gamma'}{\Gamma}\right)'(-a,\chi)+\sum_{\rho\in Z_\chi^\ntriv} \frac{1}{(a+\rho)^2},
\end{equation}
which shows that, in order to bound $(L'/L)'(-a,\chi)$, we are left with estimating $(\Gamma'/\Gamma)'(-a,\chi)$. For this, it is enough to combine the identity
$$
\left(\frac{\Gamma'}{\Gamma}\right)'(s+1)=\left(\frac{\Gamma'}{\Gamma}\right)'(s)-\frac{1}{s^2},
$$ 
obtained by differentiation of the right-hand identity of \eqref{equation: gammaperiod}, with the fact that $(\Gamma'/\Gamma)'(s)$ is a positive decreasing function for $s\in \R_{>0}$. Indeed, using these facts one finds
$$
\left|\left(\frac{\Gamma'}{\Gamma}\right)'(-a,\chi)\right|\leq \sum_ {j=1}^{\gamma_\chi}\left(\frac{\Gamma'}{\Gamma}\right)'\left(\frac{-a+\kappa_{\chi,j}}{2}+1\right)+\frac{4}{(-a+\kappa_{\chi,j})^2}=O(n_\chi).
$$
Putting together \eqref{equation: boundLma} and \eqref{equation: boundLpa}, we obtain
$$
|R_\chi|=O([k:\Q](n_\chi+\log(B_\chi)) \log(2+w_\chi)x^{-a}\log x).
$$
We can now rewrite \eqref{equation: intform} as
\begin{equation}\label{equation: firstaprox}
\sum_{\NN(\p)\leq x}\Lambda_\chi(\p, x)= \frac{16}{25}\delta(\chi)x+O([k:\Q](n_\chi+\log(B_\chi))\log(2+w_\chi)\cdot \sqrt x(\log x)^2). \qedhere
\end{equation}
\end{proof}

A moment of reflection shows that the presence of the factor $(\log x)^3$ in the error term of Proposition \ref{proposition: truncsumweak} comes exclusively from the rough bound \eqref{equation: roughbound}. The main idea of the proof of Proposition \ref{proposition: truncsum}, which is a strengthening of Proposition \ref{proposition: truncsumweak}, is to sharpen the bound \eqref{equation: roughbound} using Proposition \ref{proposition: truncsumweak} itself. We first need the following auxiliary lemma.

\begin{lemma}\label{lemma: auxdims}
Let $\chi$ be a selfdual irreducible character of $\ST(A)$. Write $\Psi^2_\chi$ for the virtual selfdual character $\chi(\cdot^2)$. Suppose that Conjecture \ref{conjecture: GRH} holds for $\chi$. Then:
$$
t_{\Psi^2_\chi}=d_\chi, \quad w_{\Psi^2_\chi}=2w_\chi, \quad \delta(\Psi^2_\chi)=\pm 1,\quad N_{\Psi^2_\chi}=2B_\chi^2,\quad \gamma_{\Psi^2_\chi}\leq 2d_\chi.\vspace{0,1cm}
$$
\end{lemma}

\begin{proof}
The first and second relations are obvious. The value $\delta(\Psi^2_\chi)$ is the Frobenius--Schur indicator of $\chi$, which is known to be $\pm 1$ when $\chi$ is selfdual. 
As for the fourth and fifth relations, note that $L(s,\Psi^2_\chi)=L(2s,\chi)$. It follows that the completed $L$-function $\Lambda(s,\Psi^2_\chi)$ is a scalar multiple of
$$
\Lambda(2s,\chi)=B_\chi^{sd_\chi} L(2s,\chi)\Gamma(2s,\chi).
$$
By Legendre's duplication formula $\Gamma(2s)=\Gamma(s)\Gamma(s+1/2)2^{2s-1} \pi^{-1/2}$, we can write
$$
\Gamma(2s,\chi)=\pi^{d_\chi s}\prod_{j=1}^{d_\chi}\left(\Gamma\left(\frac{s+\kappa_{\chi,j}/2}{2}\right)\Gamma\left(\frac{s+\kappa_{\chi,j}/2+1}{2}\right)2^{s+\kappa_{\chi,j}/2-1}\pi^{-1/2}\right),
$$
from which we see that the analytic conductor of $\Lambda(s,\Psi^2_\chi)$ must be $2B_\chi^2$ and that $\gamma_{\Psi^2_\chi}\leq 2 d_\chi$. 

We note that using $\Psi_\chi^2=2\Sym^2(\chi)-\chi^2$, we can directly derive $N_{\Psi^2_\chi}=O(B_\chi^2)$ without appealing to the analytic continuation of $\Lambda(s,\chi)$.
\end{proof}

\begin{proposition}\label{proposition: truncsum}
Let $\chi$ be a selfdual irreducible character of $\ST(A)$ and assume that Conjecture~\ref{conjecture: GRH} holds for $\chi$. For every $x\geq 2$ we have
$$
\sum_{\NN(\p)\leq x} \chi(y_\p)\log(\NN(\p))\left(\frac{\NN(\p)^r}{x}\right)^a\log\left(\frac{x}{\NN(\p)^r}\right) =\frac{16}{25}\delta(\chi)x+O\big(g\log(2|\Delta_k|)d_\chi\log(2N)\log(2+w_\chi)\cdot \sqrt x\big).
$$
\end{proposition}

\begin{proof}
Note that
$$
\sum_{r\geq 2} \sum _{\NN(\p)^r\leq x} \Lambda_\chi(\p^r,x)=\sum_{\NN(\p)\leq x^{1/2}}\Lambda_{\Psi^2_\chi}(\p,x)+\sum_{r\geq 3}\sum _{\NN(\p)\leq x^{1/r}}\Lambda_\chi(\p^r,x). 
$$ 
We will apply Proposition \ref{proposition: truncsumweak} to the first term of the sum of the right-hand side of the above equation. For the second term, we may repeat the argument of \eqref{equation: roughbound}. Indeed, using Lemma \ref{lemma: auxdims}, we obtain
$$
\begin{array}{lll}
\displaystyle{\sum_{r\geq 2} \sum _{\NN(\p)^r\leq x} \Lambda_\chi(\p^r,x)} & = & \displaystyle{O(\sqrt x)+O([k:\Q](n_\chi+\log(B_\chi))\log(2+w_\chi)\cdot x^{1/4}(\log x)^2)+}\\
 &  & \displaystyle{+O([k:\Q]n_\chi x^{1/3}(\log x)^2)}\\[8pt]
&=& \displaystyle{O(\sqrt{x})+O([k:\Q](n_\chi+\log(B_\chi))\log(2+w_\chi) \cdot x^{1/3}(\log x)^2)}.
\end{array}
$$
Rerunning the argument starting at \eqref{equation: intform} and finalizing at \eqref{equation: firstaprox} we arrive at
$$
\sum_{\NN(\p)\leq x}\Lambda_\chi(\p, x)= \frac{16}{25}\delta(\chi)x+O([k:\Q](n_\chi+\log(B_\chi))\log(2+w_\chi)\cdot \sqrt x).
$$
Using that $\log(B_\chi)=O(g\log(2|\Delta_k|)d_\chi\log(2N))$ (see Remark \ref{remark: Bchi} below) and that $d_\chi=n_\chi$ (as $\chi$ is irreducible), we obtain
$$
\sum_{\NN(\p)\leq x}\Lambda_\chi(\p, x)= \frac{16}{25}\delta(\chi)x+O(g[k:\Q]\log(2|\Delta_k|)d_\chi\log(2N)\log(2+w_\chi)\cdot \sqrt x).
$$
To conclude we need to show that we can restrict to primes not dividing $N$ as in the statement of the proposition (recall the convention for sums over the classes $y_\p$). But note that
\begin{equation} \label{eq:ramified primes}
\sum_{\p|N,\NN(\p)\leq x}\Lambda_\chi(\p, x)=O([k:\Q]d_\chi\log(2N)\log x),
\end{equation}
which is subsumed in the error term.
\end{proof}

\begin{remark}\label{remark: Bchi}
We now justify the bound $\log(B_\chi)=O(g\log(2|\Delta_k|)d_\chi\log(2N))$ used in the proof of the above proposition. By definition of $B_\chi$ we have
$$
\log(B_\chi)=d_\chi\log(|\Delta_k|)+\log(N_\chi).
$$
Since Minkowski's bound implies $[k:\Q]=O(\log(2|\Delta_k|))$, it will suffice to show that 
$$\log(N_\chi)=O(g[k:\Q]d_\chi\log(N)).$$
To show the above, we proceed as in \cite[Rem. 2.8]{BFK23}, but making the dependence on $g$ and $[k:\Q]$ explicit. Expressing $N_\chi$ as the product 
$$
N_\chi=\prod_\p\NN(\p)^{f_\chi(\p)},
$$
we see that it will be enough to show that $f_\chi(\p)=O(g[k:\Q]d_\chi)$. Let $\ell$ be a prime. There exists a number field $E$ and a prime $\lambda$ of $E$ over $\ell$ such that the representation $V_{\chi,\ell}$ is the base change to $\overline \Q_\ell$ of some representation over $E_\lambda$. Let $\Os_{E_\lambda}$ denote the ring of integers of $E_\lambda$ and let $T_{\chi,\lambda}$ denote a $\Os_{E_\lambda}$-lattice in $V_{\chi,\ell}$ stable by the action of $G_\p$. Then $f_\chi(\p)$ can be decomposed as
$$
f_\chi(\p)= \varepsilon_\chi(\p)+\delta_\chi(\p)\,,
$$
where $\varepsilon_\chi(\p)=d_\chi-\dim(V_{\chi,\ell}^{I_\p})=O(d_\chi)$ and $\delta_\chi(\p)$ is the Swan conductor of $V_\chi[\lambda]:=T_{\chi,\lambda}/\lambda T_{\chi,\lambda}$ for every $\ell$ coprime to~$\p$. The action of $G_\p$ on $V_\chi[\lambda]$ factors through the faithful action of a finite group $G_{\chi,\p}$. Let
$$
G_{\chi,\p} \supseteq G_{\chi,\p,0} \supseteq G_{\chi,\p,1} \supseteq \dots
$$
denote the normal filtration of ramification groups of $G_{\chi,\p}$. If $p$ denotes the rational prime below $\p$, and $v_p$ is the corresponding $p$-adic valuation, then \cite[Prop. 5.4]{BK94} implies that $\delta_\chi(\p)=O(v_p(\#G_{\chi,\p,1})\cdot d_\chi [k:\Q])$. By construction of $V_{\chi,\ell}$, the group $G_{\chi,\p,1}$ is a quotient of the wild inertia subgroup $G_{\ell,\p,1}\subseteq \GL_{2g}(\F_\ell)$ acting on $A[\ell]$. In particular, we must have $v_p(\#G_{\chi,\p,1})\leq v_p(\#\GL_{2g}(\F_\ell))$, and hence it will be enough to show that for every prime $p$ there is a prime $\ell$ such that $v_p(\#\GL_{2g}(\F_\ell))=O(g)$.

 If $p$ is odd and $\ell$ is chosen to be a generator of $(\Z/p^2\Z)^\times$, then by \cite[(9)]{GL06} we have
$$
v_p(\#\GL_{2g}(\F_\ell))=\left\lfloor\frac{2g}{p-1}\right\rfloor + \left\lfloor\frac{2g}{(p-1)p}\right\rfloor + \left\lfloor\frac{2g}{(p-1)p^2}\right\rfloor + \dots =O(g).
$$
If $p=2$, then take $\ell=3$. The facts that $v_2(3^n-1)=1$ if $n$ is odd and that $v_2(3^n-1)=v_2(n)+2$ if $n$ is even\footnote{To show this, proceed by induction using that $v_2(3^n+1)=1$ or $2$ depending on whether $n$ is even or odd, and that in the latter case we have $(3^n-1)=(3^{n/2}-1)(3^{n/2}+1)$.} yield
$$
v_2(\#\GL_{2g}(\F_3))=4g+\left\lfloor\frac{g}{2}\right\rfloor + \left\lfloor\frac{g}{2^2}\right\rfloor + \left\lfloor\frac{g}{2^3}\right\rfloor + \dots =O(g).
$$
\end{remark}

\section{Proof of the main theorem}\label{section: linsignrevis}

Let $A$ and $A'$ be abelian varieties defined over $k$ of dimensions $g$ and $g'$. Suppose that the Mumford--Tate conjecture holds for both $A$ and $A'$. Let 
$$\varrho\colon\ST(A)\rightarrow \GL(V)\qquad
\big(\text{resp. } \varrho'\colon \ST(A')\rightarrow \GL(V')\big)
$$ 
be the tautological representations of $\ST(A)$ (resp. $\ST(A')$), as presented in \S\ref{section: conjectures}. By \cite{Com19}, the Mumford--Tate conjecture also holds for $A\times A'$, and hence we may consider $\ST(A\times A')$. There is a natural inclusion of $\ST(A\times A')$ in $\ST(A)\times \ST(A')$, which we will denote by $\kappa$.
Consider the virtual representation 
$$
\ST(A\times A')\stackrel{\kappa}{\hookrightarrow} \ST(A)\times\ST(A')\rightarrow  \GL((V^{\oplus-2g} \oplus V\otimes V)\otimes (V'^{\oplus 2g'}\oplus V'\otimes V')),
$$
whose character we denote by $\psi$.
 
\begin{lemma}\label{lemma: trivmult}
If $\Hom(A,A')=0$, then $\delta(\psi)>0$.
\end{lemma}

\begin{proof}
It suffices to show that 
\begin{equation}\label{equation: 0dims}
\dim_{\C}(V\otimes V')^{\ST(A\times A')}=0,\qquad 
\dim_{\C}(V\otimes V'\otimes V')^{\ST(A\times A')}=0.
\end{equation}
Indeed, by symmetry we then also have $\dim_{\C}(V\otimes V\otimes V')^{\ST(A\times A')}=0$, and therefore the selfduality of $V$ and $V'$ yields
$$
\delta(\psi)=\dim_{\C}(V\otimes V\otimes V'\otimes V')^{\ST(A\times A')}=\dim_{\C}\Hom_{\ST(A\times A')}(V\otimes V',V\otimes V')>0.
$$
We now turn to \eqref{equation: 0dims}. On the one hand, by hypothesis and Faltings's isogeny theorem, we have
$$
\dim_{\C}(V\otimes V')^{\ST(A\times A')}=\dim_{\C}\Hom_{\ST(A\times A')}(V,V')=\dim_{\Q_\ell}\Hom_{G_k}(V_\ell(A),V_\ell(A'))=\mathrm{rk}_{\Z}\Hom(A,A')=0.
$$
On the other hand, let $\pi$ (resp. $\pi'$) denote the projection from $\ST(A)\times \ST(A')$ to $\ST(A)$ (resp. $\ST(A')$), and define $\theta=\varrho \circ \pi \circ \kappa$ and $\theta'=\varrho' \circ \pi' \circ \kappa$. Denoting by $\mu_{\ST(A\times A')}$ the Haar measure of $\ST(A\times A')$, we find
$$
\dim_\C(V\otimes V'\otimes V')^{\ST(A\times A')}=\int_{\ST(A\times A')}\Psi(g)\mu_{\ST(A\times A')}(g),
$$
where $\Psi(g)$ stands for $\Tr(\theta(g))\cdot \Tr(\theta'(g))^2$ for $g\in\ST(A\times A')$.
By the Hodge condition in \cite[Def. 3.1]{FKRS12}, the Sato--Tate group $\ST(A\times A')$ contains a Hodge circle, and in particular it contains minus the identity matrix. The identity $\Psi(-g)=-\Psi(g)$ then shows that the integral in the above equation is zero.
\end{proof}

\begin{proof}[Proof of Theorem \ref{theorem: Main}]
For every prime $\p$ not dividing $NN'$, let $y_\p$ denote the conjugacy class in $\ST(A\times A')$ attached to $\Frob_\p$, as presented in \S\ref{section: conjectures}. The proof is based on the following trivial observation. Since 
$$
\psi(y_\p)=\overline a_\p(A) \overline a_\p(A') (-2g +\overline a_\p(A)) (2g' +\overline a_\p(A')),
$$
we have that 
$$
\psi(y_\p) >0 \qquad \text{if and only if}\qquad a_\p(A)\cdot a_\p(A')<0. 
$$
Note that $w_\psi=2$ and that $d_\psi\leq 64 g^2 (g')^2$. Applying Proposition~\ref{proposition: truncsum} to the virtual character $\psi$, we obtain
\begin{align}\label{equation: finineq}
\sum_{\NN(\p)\leq x}\psi(y_\p)\log(\NN(\p))\left(\frac{\NN(\p)}{x} \right)^{1/4}\log\left(\frac{x}{\NN(\p)}\right) & =\frac{16}{25} \delta(\psi) x \nonumber \\ &+ O\left([k:\Q]\log(2|\Delta_k|)g^2(g')^2(g+g')\log(2NN')\sqrt x \, \right),
\end{align}
where $\delta(\psi) > 0$ by Lemma~\ref{lemma: trivmult}. Therefore, there exists a positive constant $C$ such that if 
$$
x=C[k:\Q]^2\log(2|\Delta_k|)^2g^4(g')^4(g+g')^2\log(2NN')^2,
$$ 
then the right-hand side of \eqref{equation: finineq} is strictly positive. But this means that there exists $\p$ with 
$$
\NN(\p)=O([k:\Q]^2\log(2|\Delta_k|)^2g^4(g')^4(g+g')^2\log(2NN')^2)
$$ 
such that $\psi(y_\p)>0$.  

\end{proof}

\end{document}